\documentclass[12pt]{amsart}

\usepackage[utf8]{inputenc}
\usepackage{utopia}
\usepackage{graphicx}
\usepackage{float}
\usepackage{caption}
\usepackage{subcaption}
\usepackage{scalerel,nicefrac}
\usepackage[dvipsnames]{xcolor}

\usepackage{geometry}
\newtheorem{theorem}{Theorem}[section]
\newtheorem{lemma}[theorem]{Lemma}
\newtheorem{thmx}{Theorem}

\newtheorem{proposition}[theorem]{Proposition}
\newtheorem{claim}[theorem]{Claim}

\theoremstyle{definition}
\newtheorem{definition}[theorem]{Definition}

\theoremstyle{remark}
\newtheorem{remark}[theorem]{Remark}
\newtheorem{que}[theorem]{Question}

\numberwithin{equation}{section}

\everymath{\displaystyle}

\begin{document}

\title{On Lagrangian Tori in $S^2\times S^2$}


\author{Han Lou}
\address{Department of Mathematics, University of Georgia, Athens, GA 30602}
\email{Han.Lou@uga.edu}



\begin{abstract}
In \cite{FOOO}, K. Fukaya, Y. Oh, H. Ohta, and K. Ono (FOOO) obtained the monotone symplectic manifold $S^2\times S^2$ by resolving the singularity of a toric degeneration of a Hirzebruch surface. They identified a continuum of toric fibers in the resolved toric degeneration that are  not Hamiltonian isotopic to the toric fibers of the standard toric structure on $S^2\times S^2$. In this paper, we provide a comprehensive classification: for any toric fiber in FOOO's construction of $S^2\times S^2$, we determine whether it is Hamiltonian isotopic to a toric fiber of the standard toric structure of $S^2\times S^2$. 
\end{abstract}

\maketitle

\tableofcontents

\begin{section}{Introduction}
Let $(M^{2n}, \omega)$ be a symplectic manifold. A \textit{Lagrangian submanifold} $L$ is a submanifold of $M$ with dimension $n$ such that $\omega|_L=0$. In \cite{A}, Arnold defined a Lagrangian knot as a connected component of the space of Lagrangian embeddings in a fixed symplectic manifold. For Lagrangian embeddings $\mathbb{R}^2\to\mathbb{R}^4$(coinciding with embeddings of the plane $(z_1, 0)$ outside some sphere in the standard four-dimensional symplectic space), he also proposed the following two questions.
\begin{que}{\cite[section 6]{A}}
    Can any knot in the ordinary sense be realized by a Lagrangian one?
\end{que}
\begin{que}{\cite[section 6]{A}}
     Are there purely Lagrangian knots, that is, Lagrangian embeddings homotopic to the plane in the class of all embeddings, but non-homotopic in the class of Lagrangian embeddings?
\end{que}

Buliding on Arnold's questions, several significant works have deepened our understanding of Lagrangian embeddings.
In \cite{C}, Y. Chekanov constructed the \textit{special tori} in $\mathbb{R}^{2n}$ that are not symplectomorphic to each other. These tori are examples of monotone Lagrangian tori that are Lagrangian isotopic but not Hamiltonian isotopic to an elementary torus.
In \cite{EP}, Y. Eliashberg and L. Polterovich considered if two Lagrangian embeddings are isotopic in smooth, Lagrangian, or Hamiltonian sense. 

Considering the monotone $S^2\times S^2$, G. Dimitroglou Rizell, E. Goodman, and A. Ivrii in \cite{RGI} showed that any two Lagrangian tori are Lagrangian isotopic. There are several different constructions of monotone Lagrangian tori in $S^2\times S^2$ that are not Hamiltonian isotopic to the Clifford torus, the product of the equators. Using P. Biran's circle bundle construction in \cite{Bi} one can get such a Lagrangian torus. P. Albers and U. Frauenfelder in \cite{AF} constructed a nondisplaceable Lagrangian torus in $T^*S^2$. Then one can get such a Lagrangian torus by an embedding from $D^*S^2$, a disk subbundle of $T^*S^2$, to $S^2\times S^2$. 
M. Entov and L. Polterovich constructed a non-heavy monotone Lagrangian torus in \cite[Example 1.22]{EP2}.
Y. Chekanov and F. Schlenk in \cite{CS} also constructed such a monotone Lagrangian torus. J. Oakley and M. Usher in \cite{OU} showed that the above four Lagrangian tori are Hamiltonian isotopic to each other. A. Gadbled in \cite{G} also showed that the Lagrangian tori in \cite{CS} and \cite{Bi} are Hamiltonian isotopic. 

We consider $S^2$ as the unit sphere in $\mathbb{R}^3$ with symplectic form $\omega_{std}$ such that the $\omega_{std}$-area of $S^2$ is $4\pi$. Then $\left(S^2, \frac{1}{2}\omega_{std}\right)\times \left(S^2, \frac{1}{2}\omega_{std}\right)$ has a standard toric structure with the moment map
\begin{equation*}
    \mu:S^2\times S^2\to \mathbb{R}^2\\
\end{equation*}
\begin{equation*}
    \left(\left(v_1, v_2, v_3\right), \left(w_1, w_2, w_3\right)\right)\mapsto\left(\frac{1}{2}v_1, \frac{1}{2}w_1\right)
\end{equation*}
The moment polytope, as in Figure \ref{p1}, is the square
\begin{equation*}
    P_1=\left\{(x, y)\in\mathbb{R}^2\mid -\frac{1}{2}\le x\le\frac{1}{2}, -\frac{1}{2}\le y\le\frac{1}{2}\right\}
\end{equation*}

\begin{figure}[h]
    \centering
    \includegraphics[scale=0.5]{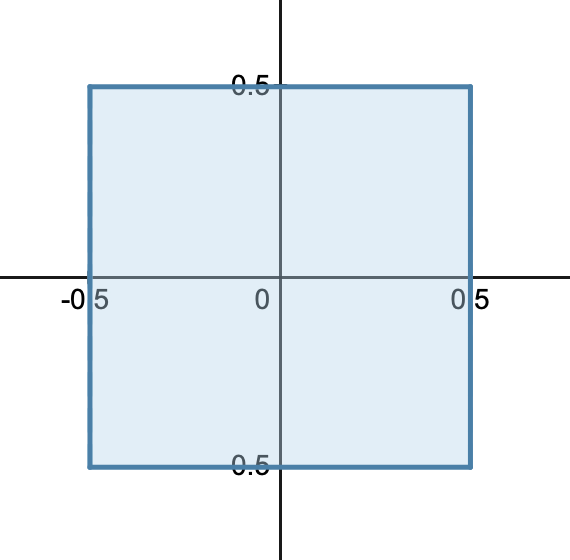}
    \caption{The moment polytope $P_1$ for the standard toric structure of $S^2\times S^2$.}
    \label{p1}
\end{figure}

As in \cite{FOOO}, $S^2\times S^2$ can be obtained by resolving the singularity of a toric degeneration. Now we recall the construction.
The toric Hirzebruch surfaces $F_2(\alpha)$, $0<\alpha<1$, are toric manifolds with moment polytope 
\begin{equation*}
    \{(x, y)\in\mathbb{R}^2\mid 0\le x\le 2-2y, 0\le y\le 1-\alpha\}
\end{equation*}
As $\alpha\to 0$, we obtain an orbifold $F_2(0)$ with a singularity of the form $\mathbb{C}^2/\pm$. The moment polytope of $F_2(0)$, as in Figure \ref{p2}, is 
\begin{equation*}
    P_2=\{(x, y)\in\mathbb{R}^2\mid 0\le x\le 2-2y, y\ge 0 \}
\end{equation*}
\begin{figure}[h]
    \centering
    \includegraphics[scale=0.5]{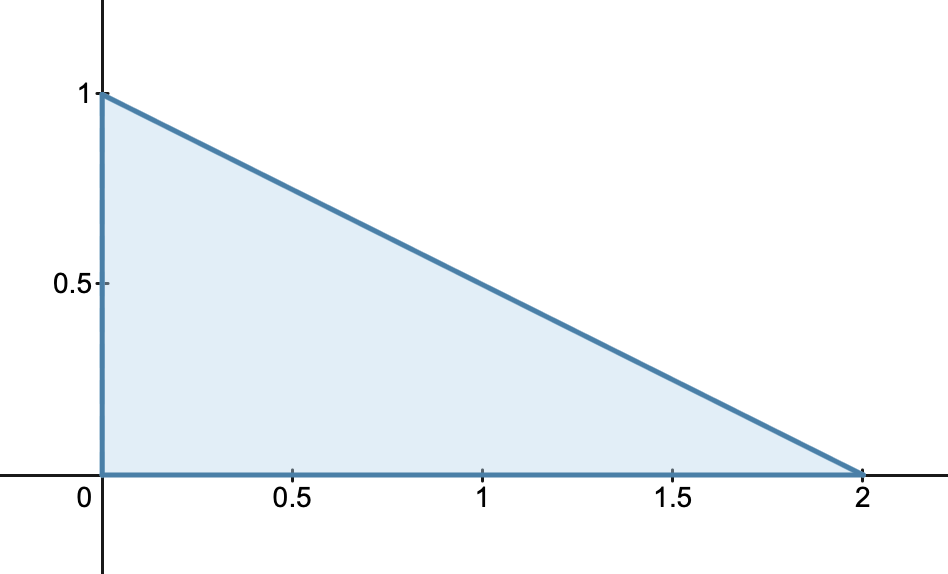}
    \caption{The moment polytope of $\hat{F}(0)$.}
    \label{p2}
\end{figure}
and the preimage of the point $(0, 1)$ is the singularity. To resolve the singularity, we replace a neighborhood of it with a neighborhood of the zero section of the cotangent bundle $T^*S^2$. The resulting complex surface is denoted  by $\hat{F}(0)$. There is a well-known result that $F_2(\alpha)$ is symplectomorphic to $\left(S^2, \frac{1-\alpha}{2}\omega_{std}\right)\times\left(S^2, \frac{1+\alpha}{2}\omega_{std}\right)$, and $\hat{F}(0)$ is symplectomorphic to $\left(S^2, \frac{1}{2}\omega_{std}\right)\times\left(S^2, \frac{1}{2}\omega_{std}\right)$ \cite[Proposition 5.1]{FOOO}. We still say the moment polytope of $\hat{F}(0)$ is $P_2$ with the preimage of the point $(0, 1)$ being $S^2$.

Given a point $(x, y)$ in the segment $\left\{(x, y)\in P_2\mid x+y=1,\frac{1}{2}\le y<1\right\}$, K. Fukaya, Y. Oh, H. Ohta, and K. Ono in \cite{FOOO} have shown the preimage of $(x, y)$ under the moment map is a nondisplaceable Lagrangian torus using Lagrangian Floer homology. In particular, the preimage is not Hamiltonian isotopic to a toric fiber of the standard toric structure. This conclusion can be obtained from the fact that the only nondisplaceable toric fiber of the standard toric structure is the Clifford torus and from Theorem 1.1 in \cite{FOOO}. J. Oakley and M. Usher in \cite{OU} showed that the preimage of $\left(\frac{1}{2}, \frac{1}{2}\right)$ is Hamiltonian isotopic to the Lagrangian torus in \cite{Bi}, \cite{AF}, \cite{EP2}, and \cite{CS}. Then one can ask the following questions.
\begin{que}
    Can this discussion be extended to encompass other interior points within the moment polytope $P_2$? Specifically, can we show if their preimages are Hamiltonian isotopic to toric fibers of the standard toric structure?
\end{que}
\begin{que}
    Are there other Lagrangian tori, except those in \cite{FOOO}, that are not Hamiltonian isotopic to the product tori in $S^2\times S^2$? We call a torus of the form $\alpha\times\beta$ with $\alpha$ and $\beta$ are embedded curves in $S^1\times\{pt\}$ and $\{pt\}\times S^1$ respectively a product torus.
\end{que}

We will prove the following theorems to answer the questions.
\begin{thmx}\label{A}
    Given an interior point $(x, y)$ in $P_2$ with $x+y\ne 1$, let $L(x, y)$ be the preimage of $(x, y)$ in $\hat{F}(0)$ and still denote by $L(x, y)$ its image under a symplectomorphism from $\hat{F}(0)$ to $S^2\times S^2$. Then $L(x, y)$ is a Lagrangian torus Hamiltonian isotopic to a toric fiber. Furthermore, if we denote the preimage of an interior point $(\xi, \zeta)$ in $P_1$ by $T(\xi,\zeta)$, then
    \begin{equation*}
        L(x, y) \:\text{is Hamiltonian isotopic to} \begin{cases}
        
            T\left(\frac{1}{2}-y, \frac{3}{2}-2y-x\right) & \text{for}\: 1-y<x< 2-2y\\
            T\left(-\frac{1}{2}+x, \frac{1}{2}-y\right) & \text{for}\: 0<x<1-y
        \end{cases}
    \end{equation*}
\end{thmx}

\begin{remark}
    Theorem \ref{A} does not depend on the choice of the symplectomorphism between $\hat{F}(0)$ and $S^2\times S^2$ since every symplectomorphism of $S^2\times S^2$ can be written as the composition of a Hamiltonian diffeomorphism and the diffeomorphism that switches the two factors of $S^2\times S^2$. See \cite{OU} and \cite{Gro}. In addition, the toric fiber $T(\xi, \zeta)$ is Hamiltonian isotopic to $T(\zeta, \xi)$ by \cite{B}.
\end{remark}

\begin{thmx}\label{B}
    Let $(x, y)$ be an interior point of $P_2$ and $x+y=1$. Then $L(x, y)$ is not Hamiltonian isotopic to a product torus.
\end{thmx}

\begin{remark}
    For an interior point $(x, y)$ in $P_2$ with $x+y=1$ and $0<y<\frac{1}{2}$, the preimage of $(x, y)$ is a displaceable Lagrangian torus, so Theorem \ref{B} does not follow from the  techniques in \cite{FOOO}.
\end{remark}

\textbf{Convention} Given a smooth function $H$ on a symplectic manifold $(M, \omega)$, we denote the Hamiltonian vector field by $X_H$ and $w(X_H, \cdot)=-dH$. The flow of $X_H$ is denoted by $\phi_H^t$ and $\phi_H^0=id$.

    \textbf{Acknowledgement} I am grateful to my advisor, Michael Usher, for many helpful discussions on this work and consistent support on my Ph.D. journey. The seed of this paper was planted in the Fall 2023 Informal Student Symplectic Seminar. I am also grateful to Shengzhen Ning for organizing the seminar and to all my friends for the inspiring talks. 
\end{section}

\begin{section}{Proof of Theorem \ref{A}}

J. Oakley and M. Usher gave an explicit expression of the symplectomorphism from $\hat{F}(0)$ to $\left(S^2, \frac{1}{2}\omega_{std}\right)\times\left(S^2, \frac{1}{2}\omega_{std}\right)$ in \cite[Proof of Proposition 2.1]{OU}. The image of $L(x, y)$ under this symplectomorphism, still denoted by $L(x, y)$, is 
\begin{align*}
    L(x, y) &=\left\{(v, w)\in S^2\times S^2 \Bigm| \frac{1}{2}|v+w|+\frac{1}{2}(v+w)\cdot e_1=x, 1-\frac{1}{2}|v+w|=y\right\}\\
    &=\left\{(v, w)\in S^2\times S^2\mid v_1+w_1=2(x+y-1), v\cdot w=2(1-y)^2-1\right\}
\end{align*}
where $(e_1, e_2, e_3)$ is an orthonormal basis for $\mathbb{R}^3$.
Then we change the coordinates by 
\begin{align*}
    p &=x+y-1\\
    q &=1-y
\end{align*}
Under the coordinates $(p, q)$, the moment polytope $P_2$ becomes
\begin{equation*}
    \left\{(p, q)\in\mathbb{R}^2\mid -q\le p\le q, q\le 1\right\}
\end{equation*}
and we still denote it by $P_2$. The Lagrangian torus $L(x, y)$ can be written as 
\begin{equation*}
    L_1(p, q)=\left\{(v, w)\in S^2\times S^2\mid v_1+w_1=2p, v\cdot w=2q^2-1\right\}
\end{equation*}

Similar to \cite[Proof of Proposition 2.4]{OU}, we can prove the following result.
\begin{proposition}\label{prop00}
    For $(p, q)\in\text{Int}(P_2)$ the Lagrangian torus $L_1(p, q)$ is the orbit of an embedded curve $\Gamma_1$ in $S^2\times S^2$ under the $S^1$-action $(R_t, R_t)$, where
    \begin{center}
    $R_t=\begin{bmatrix}
        1 & 0 & 0\\
        0 & \cos(t) & -\sin(t)\\
        0 & \sin(t) & \cos(t)
    \end{bmatrix}$
    \end{center}
    is the rotation around $e_1$-axis by angle $t$.
\end{proposition}
\begin{proof}
    Note that 
    \begin{equation*}
        \left(v_0, w_0\right)=\left(\left(p, \sqrt{1-q^2}, \sqrt{q^2-p^2}\right), \left(p, -\sqrt{1-q^2}, \sqrt{q^2-p^2}\right)\right)
    \end{equation*}
    is a point in $L_1(p, q)$. Then we rotate $v_0$ and $w_0$ around the vector $v_0+w_0$ by angle $\theta$ to get an embedded curve  
    \begin{align*}
        \Gamma_1 =\left\{\left(\begin{bmatrix}
            p+\frac{\sqrt{q^2-p^2}\sqrt{1-q^2}}{q}\sin(\theta)\\
            \sqrt{1-q^2}\cos(\theta)\\
            \sqrt{q^2-p^2}-\frac{p\sqrt{1-q^2}}{q}\sin(\theta)
        \end{bmatrix}, \begin{bmatrix}
            p-\frac{\sqrt{q^2-p^2}\sqrt{1-q^2}}{q}\sin(\theta)\\
            -\sqrt{1-q^2}\cos(\theta)\\
            \sqrt{q^2-p^2}+\frac{p\sqrt{1-q^2}}{q}\sin(\theta)
        \end{bmatrix}\right)\;\middle|\: \theta\in[0, 2\pi]\right\}
    \end{align*}

    Now we consider the Hamiltonians 
    \begin{align*}
        \left(F_1, F_2\right): \left(S^2, \frac{1}{2}\omega_{std}\right)\times \left(S^2, \frac{1}{2}\omega_{std}\right) &\to \mathbb{R}^2\\
        (v, w) &\mapsto \left(-(v+w)\cdot e_1, \frac{v\cdot w}{4q}\right)
    \end{align*}
    Then $L_1(p, q)$ is the regular level set $\left(F_1, F_2\right)^{-1}(-2p, \frac{2q^2-1}{4q})$.
    The Hamiltonian vector field $X_{F_2}$ of $F_2$ is 
    \begin{equation*}
        X_{F_2}(v, w)=\left(\frac{v\times w}{2q}, \frac{w\times v}{2q}\right)
    \end{equation*}
    Take $(v, w)\in\Gamma_1$. Then one can compute 
    \begin{align*}
        \left(\frac{v\times w}{2q}, \frac{w\times v}{2q}\right) &=\left(\begin{bmatrix}
            \frac{\sqrt{q^2-p^2}\sqrt{1-q^2}}{q}\cos(\theta)\\
            -\sqrt{1-q^2}\sin(\theta)\\
            -\frac{p\sqrt{1-q^2}}{q}\cos(\theta)
        \end{bmatrix}, \begin{bmatrix}
            -\frac{\sqrt{q^2-p^2}\sqrt{1-q^2}}{q}\cos(\theta)\\
            \sqrt{1-q^2}\sin(\theta)\\
            \frac{p\sqrt{1-q^2}}{q}\cos(\theta)
        \end{bmatrix}\right)\\
        &=\frac{d\gamma_1}{d\theta}
    \end{align*}
    where $\gamma_1$ is a parametrization of $\Gamma_1$ such that 
    \begin{equation*}
        \gamma_1(\theta)=\left(\begin{bmatrix}
            p+\frac{\sqrt{q^2-p^2}\sqrt{1-q^2}}{q}\sin(\theta)\\
            \sqrt{1-q^2}\cos(\theta)\\
            \sqrt{q^2-p^2}-\frac{p\sqrt{1-q^2}}{q}\sin(\theta)
        \end{bmatrix}, \begin{bmatrix}
            p-\frac{\sqrt{q^2-p^2}\sqrt{1-q^2}}{q}\sin(\theta)\\
            -\sqrt{1-q^2}\cos(\theta)\\
            \sqrt{q^2-p^2}+\frac{p\sqrt{1-q^2}}{q}\sin(\theta)
        \end{bmatrix}\right)
    \end{equation*}
    Thus $\Gamma_1$ is an integral curve of $X_{F_2}$. Since $\{F_1, F_2\}=0$, then $L_1(p, q)$ is the orbit of the curve $\Gamma_1$ under the flow of $F_1$, which gives the $S^1$-action $\left(R_t, R_t\right)$.
\end{proof}

As in \cite[Proof of Proposition 2.4]{OU} and \cite[Lemma 2.4]{G}, the action $\left(R_t, R_t\right)$ and $\left(R_t, R_{-t}\right)$ are conjugate in $SO(3)\times SO(3)$, i.e. 
\begin{equation*}
    \left(R_t, R_t\right)=\left(D_1, D_2\right)^{-1}\left(R_t, R_{-t}\right)\left(D_1, D_2\right)
\end{equation*}
where $D_1=\begin{bmatrix}
    1 & 0 & 0\\
    0 & 1 & 0\\
    0 & 0 & 1
\end{bmatrix}$ and $D_2=\begin{bmatrix}
    -1 & 0 & 0\\
    0 & -1 & 0\\
    0 & 0 & 1
\end{bmatrix}$.
Define $\Gamma_2=\left(D_1, D_2\right)\Gamma_1$ and $L_2(p, q)$ to be the orbit of $\Gamma_2$ under the $S^1$-action $\left(R_t, R_{-t}\right)$. Note that $\left(D_1, D_2\right)L_1(p, q)=L_2(p, q)$ and $\left(D_1, D_2\right)$ is a Hamiltonian diffeomorphism. Thus we have the following result.

\begin{proposition}\label{prop0}
    $L_1(p, q)$ is Hamiltonian isotopic to $L_2(p, q)$.
\end{proposition}

\begin{subsection}{Case 1: $0<p^2< q^4$}

In this case neither of the components of $\Gamma_2$ passes through $-e_1$.

As in \cite[Proof of Proposition 2.4]{OU}, we consider the symplectomorphism
\begin{align*}
    \psi_{-1}: \left(B^2(1), 2dx\wedge dy\right) &\to\left(S^2\backslash\{-e_1\}, \frac{1}{2}\omega_{std}\right)\\
    re^{i\theta} &\mapsto\left(1-2r^2, 2r\sqrt{1-r^2}\cos(\theta), 2r\sqrt{1-r^2}\sin(\theta)\right)
\end{align*}
where $B^2(1)$ is the open ball in $\mathbb{C}$ with radius $1$.
Note that $\psi_{-1}(e^{it}\cdot re^{i\theta})=R_t\psi_{-1}(re^{i\theta})$.
Let $\tilde{\Gamma}_2=\left(\psi_{-1}\times\psi_{-1}\right)^{-1}\left(\Gamma_2\right)$.
Then $L_2(p, q)$ is symplectomorphic to the Lagrangian torus $\tilde{L}_2(p, q)$ in $\left(B^2(1), 2dx\wedge dy\right)\times\left(B^2(1), 2dx\wedge dy\right)$, that is the orbit of the curve $\tilde{\Gamma}_2$ under the $S^1$-action $\left(e^{it}, e^{-it}\right)$. 

Now we describe $\tilde{L}_2(p, q)$ in the way in \cite{EP}. Note that the $S^1$-action $\left(e^{it}, e^{-it}\right)$ is the Hamiltonian flow of the function
\begin{align*}
    H: \left(\mathbb{C}^2, idz_1\wedge d\bar{z}_1+idz_2\wedge d\bar{z}_2\right) &\to\mathbb{R}\\
    (z_1, z_2) &\mapsto |z_1|^2-|z_2|^2
\end{align*}

\begin{claim}
    $H(\tilde{\Gamma}_2)=-p$.
\end{claim}
\begin{proof}
Consider 
\begin{align*}
    h: S^2\times S^2 &\to\mathbb{R}\\
    \left((v_1, v_2, v_3), (w_1, w_2, w_3)\right) &\mapsto -\frac{1}{2}(v_1-w_1)
\end{align*}
Then $h\circ\left(\psi_{-1}\times\psi_{-1}\right)=H$. Take $(v, w)\in\Gamma_2$. Then $v_1=p+\frac{\sqrt{q^2-p^2}\sqrt{1-q^2}}{q}\sin(\theta)$ and $w_1=-p+\frac{\sqrt{q^2-p^2}\sqrt
{1-q^2}}{q}\sin(\theta)$. Thus $h(v, w)=-p$, furthermore, $H(\tilde{\Gamma}_2)=h(\Gamma_2)=-p$
\end{proof}
 Consider the function 
\begin{align*}
    F: \mathbb{C}^2 &\to\mathbb{C}\\
    (z_1, z_2) &\mapsto z_1z_2
\end{align*}
Let $\Gamma=F(\tilde{\Gamma}_2)$. 
It is easy to see that $\tilde{L}_2(p, q)\subset F^{-1}(\Gamma)\cap H^{-1}(\{-p\})$. By \cite[Lemma 4.2 A]{EP}, $F^{-1}(\Gamma)\cap H^{-1}(\{-p\}))$ is a Lagrangian torus since $\Gamma$ is an embedded curve in $B^2(1)$ and $p\ne 0$. Thus $\tilde{L}_2(p, q)= F^{-1}(\Gamma)\cap H^{-1}(\{-p\})$.

\begin{remark}
    Since the radius of $B^2(1)$ is $1$, we can restrict $F$ to $B^2(1)\times B^2(1)\to B^2(1)$.
\end{remark}

Given a function $K:\left(B^2(1), 2dx\wedge dy\right)\to\mathbb{R}$, one can show that, for any $(z_1, z_2)\in B^2(1)\times B^2(1)$,
\begin{equation*}
    dF(X_{K\circ F}(z_1, z_2))=\frac{1}{2}(|z_1|^2+|z_2|^2)\left(-\frac{\partial K}{\partial y}\frac{\partial}{\partial x}+\frac{\partial K}{\partial x}\frac{\partial}{\partial y}\right)
\end{equation*}
by direct computation. Since $K\circ F$ is invariant under the flow of $X_H$, we have $dH(X_{K\circ F})=0$. Thus we can restrict $X_{K\circ F}$ to $H^{-1}(\{-p\})$. Then
\begin{equation*}
    dF\left(X_{K\circ F}|_{H^{-1}(\{-p\})}\right)=\frac{1}{2}\sqrt{p^2+4|z_1z_2|^2}\left(-\frac{\partial K}{\partial y}\frac{\partial}{\partial x}+\frac{\partial K}{\partial x}\frac{\partial}{\partial y}\right)
\end{equation*}
Define the vector field $V^{p, K}$ on $B^2(1)$ by
\begin{equation*}
    V^{p, K}(z)=\frac{1}{2}\sqrt{p^2+4|z|^2}\left(-\frac{\partial K}{\partial y}\frac{\partial}{\partial x}+\frac{\partial K}{\partial x}\frac{\partial}{\partial y}\right)
\end{equation*}
We have 
\begin{equation}\label{der}
    dF(X_{K\circ F}(z_1, z_2))=V^{p, K}(F(z_1, z_2)),\:\text{for}\: (z_1, z_2)\in H^{-1}(\{-p\})
\end{equation}

\begin{lemma}\label{lemma0}
    Let $\phi^{t, p, K}$ be the flow of $V^{p, K}$. Then $\phi_{K\circ F}^t(\tilde{L}_2(p, q)) =F^{-1}\left(\phi^{t, p, K}(\Gamma)\right)\cap H^{-1}(\{-p\})$.
\end{lemma}
\begin{proof}
    By equation \ref{der}, we have $F\circ\phi_{K\circ F}^t(z_1, z_2)=\phi^{t, p, K}\circ F(z_1, z_2)$ for $(z_1, z_2)\in H^{-1}(\{-p\})$. Take $(z_1, z_2)\in\tilde{L}_2(p, q)$. Then $F(z_1, z_2)\in\Gamma$ and $F\circ\phi_{K\circ F}^t(z_1, z_2)\in\phi^{t, p, K}(\Gamma)$. Thus $\phi_{K\circ F}^t(z_1, z_2)\in F^{-1}\left(\phi^{t, p, K}(\Gamma)\right)$.
    Since $H\left(\phi_{K\circ F}^t(z_1, z_2)\right)=H(z_1, z_2)$, then $\phi_{K\circ F}^t(z_1, z_2)\in H^{-1}(\{-p\})$. Thus $\phi_{K\circ F}^t(\tilde{L}_2(p, q))\subset F^{-1}\left(\phi^{t, p, K}(\Gamma)\right)\cap H^{-1}(\{-p\})$. 
    
    Since $-p\ne 0$, $F^{-1}\left(\phi^{t, p, K}(\Gamma)\right)\cap H^{-1}(\{-p\})$ is a Lagrangian torus for each $t$ by \cite{EP}. Then $\phi_{K\circ F}^t$ is an embedding from torus $\tilde{L}_2(p, q)$ to torus $F^{-1}\left(\phi^{t, p, K}(\Gamma)\right)\cap H^{-1}(\{-p\})$. Thus $\phi_{K\circ F}^t(\tilde{L}_2(p, q)) =F^{-1}\left(\phi^{t, p, K}(\Gamma)\right)\cap H^{-1}(\{-p\})$.
\end{proof}

\begin{proposition}\label{prop1}
     There is a smooth function $K$ such that $\phi^{1, p, K}(\Gamma)=S^1(r)$ for some $r$ where $S^1(r)=\{re^{i\theta}\in\mathbb{C}\mid 0\le\theta\le2\pi\}$. 
\end{proposition}
\begin{proof}
    We consider the symplectic form 
    \begin{equation*}
        \omega^p=\frac{2r}{\sqrt{p^2+4r^2}}dr\wedge d\phi=\frac{2}{\sqrt{p^2+4x^2+4y^2}}dx\wedge dy
    \end{equation*} on $B^2(1)$.
    Since we assume $x+y\ne 1$, i.e. $p\ne 0$ in Theorem \ref{A}, $\omega^p$ is defined at $(0, 0)$. One can show that, for any $K$, the vector field $V^{p, K}$ is the Hamiltonian vector field of $K$ under the symplectic form $\omega^p$. Then we choose $r$ such that $\Gamma$ and $S^1(r)$ enclose the same $\omega^p$-area. This implies that $\Gamma$ and $S^1(r)$ are Hamiltonian isotopic in $\left(B^2(1), \omega^p\right)$. Thus there is a Hamiltonian $K$ such that $\phi^{1, p, K}(\Gamma)=S^1(r)$.
\end{proof}

\begin{lemma}\label{lemma1}
    The Lagrangian torus $\left(\psi_{-1}\times\psi_{-1}\right)\left(F^{-1}(S^1(r))\cap H^{-1}(\{-p\})\right)$ is the toric fiber 
    \begin{equation*}
    T\left(\frac{1+p-\sqrt{p^2+4r^2}}{2}, \frac{1-p-\sqrt{p^2+4r^2}}{2}\right).
    \end{equation*}
\end{lemma}

\begin{proof}
    Let $(\xi, \zeta)$ with $\xi=\frac{1+p-\sqrt{p^2+4r^2}}{2}$ and $\zeta=\frac{1-p-\sqrt{p^2+4r^2}}{2}$ be an interior point in $P_1$,  the moment polytope of the standard toric structure on $S^2\times S^2$. Then the toric fiber over $(\xi, \zeta)$ is 
    \begin{equation*}
        T(\xi,\zeta)=\left\{\left(\begin{bmatrix}
            2\xi\\
            \sqrt{1-4\xi^2}\cos(\theta_1)\\
            \sqrt{1-4\xi^2}\sin(\theta_1)
        \end{bmatrix},
        \begin{bmatrix}
            2\zeta\\
            \sqrt{1-4\zeta^2}\cos(\theta_2)\\
            \sqrt{1-4\zeta^2}\sin(\theta_2)
        \end{bmatrix}
        \right)\in S^2\times S^2\:\middle|\: 0\le\theta_1, \theta_2\le 2\pi\right\}
    \end{equation*}
    Then
    \begin{equation*}
    \left(\psi_{-1}\times\psi_{-1}\right)^{-1}(T(\xi, \zeta))=\left\{\left(\sqrt{\frac{1-2\xi}{2}}e^{i\theta_1}, \sqrt{\frac{1-2\zeta}{2}}e^{i\theta_2}\right)\in\mathbb{C}^2\:\middle|\:  0\le\theta_1, \theta_2\le 2\pi\right\}
    \end{equation*}
    which is the fiber over $(1-2\xi, 1-2\zeta)$ under the moment map of $\mathbb{C}^2$.

    One can easily check that $\left(\psi_{-1}\times\psi_{-1}\right)^{-1}(T(\xi, \zeta))\subset F^{-1}(S^{1}(r))\cap H^{-1}(\{-p\})$.

    Let $(z_1, z_2)$ be a point in $F^{-1}(S^{1}(r))\cap H^{-1}(\{-p\})$. Write $z_j$ as $r_je^{i\theta_j}$ for $j=1, 2$. Then $r_1r_2=r$ and $r_1^2-r_2^2=-p$. We can solve $r_1=\sqrt{\frac{-p+\sqrt{p^2+4r^2}}{2}}=\sqrt{\frac{1-2\xi}{2}}$ and $r_2=\sqrt{\frac{p+\sqrt{p^2+4r^2}}{2}}=\sqrt{\frac{1-2\zeta}{2}}$. Thus $(z_1, z_2)\in \left(\psi_{-1}\times\psi_{-1}\right)^{-1}(T(\xi, \zeta))$.
\end{proof}

\begin{proposition}
    The Lagrangian torus $L_1(x, y)$ is Hamiltonian isotopic to a toric fiber of the standard toric structure on $S^2\times S^2$.
\end{proposition}
\begin{proof}
    By Lemma \ref{lemma0} and Proposition \ref{prop1}, $\tilde{L}_{2}(p, q)$ is Hamiltonian isotopic to $F^{-1}(S^1(r))\cap H^{-1}(\{-p\})$ by $\phi_{K\circ F}^{t}$ where $K$ is a Hamiltonian such that $\Gamma$ and $S^1(r)$ are Hamiltonian isotopic in $\left(B^2(1), \omega^p\right)$. Then $L_2(p, q)=(\psi_{-1}\times\psi_{-1})\left(\tilde{L}_2(p, q)\right)$ is Hamiltonian isotopic to $\left(\psi_{-1}\times\psi_{-1}\right)\left(F^{-1}(S^1(r))\cap H^{-1}(\{-p\})\right)$ that is a toric fiber by Lemma \ref{lemma1}. By Proposition \ref{prop0}, we have $L_1(x, y)=L_1(p, q)$ is Hamiltonian isotopic to a toric fiber.
\end{proof}

Next we are going to figure out which fiber $L_1(x, y)$ is Hamiltonian isotopic to. By Proposition \ref{prop1} $\Gamma$ and $S^1(r)$ should enclose the same $\omega^p$-area where $\omega^p=\frac{2r}{\sqrt{p^2+4r^2}}dr\wedge d\phi$.

\begin{proposition}
    The $\omega^p$-area enclosed by $\Gamma$ is $2\pi-2\pi q$ for $0<p^2<q^4$.
\end{proposition}
\begin{proof}
    See Appendix.
\end{proof}

\begin{proposition}\label{use}
    The Lagrangian torus $L_1(p, q)$ is Hamiltonian isotopic to the toric fiber
    \begin{equation*}
        T(\xi, \zeta)=\begin{cases}
            T\left(q-\frac{1}{2}, q-p-\frac{1}{2}\right) & \text{for}\: 0<p<q^2\\
            T\left(p+q-\frac{1}{2}, q-\frac{1}{2}\right) & \text{for}\: -q^2<p<0
        \end{cases}
    \end{equation*}
\end{proposition}
\begin{proof}

    The explicit expression of the curve $\Gamma$ is given in Appendix. Then we note that $\Gamma$ rotates clockwise as $\theta$ changes from $0$ to $2\pi$. Thus we choose $S^1(r)$ to rotate clockwise. Then the $\omega^p$-area of $S^1(r)$ is $\pi|p|-\pi\sqrt{p^2+4r^2}$. Since $\Gamma$ and $S^1(r)$ have the same $\omega^p$-area, we have 
    \begin{equation*}
        r^2=(q-1)^2-|p|(q-1)
    \end{equation*}
    According to Lemma \ref{lemma1} $L_1(p, q)$ is Hamiltonian isotopic to the fiber over the point $(\xi, \zeta)$ where $\xi=\frac{1+p-\sqrt{p^2+4r^2}}{2}$ and $\zeta=\frac{1-p-\sqrt{p^2+4r^2}}{2}$. Note that $0<q<1$. We have 
    \begin{equation*}
        \xi=\begin{cases}
            q-\frac{1}{2} & \text{for}\: 0<p<q^2\\
            p+q-\frac{1}{2} & \text{for}\: -q^2<p<0
        \end{cases}\quad \text{and}\quad 
        \zeta=\begin{cases}
            q-p-\frac{1}{2} & \text{for}\: 0<p<q^2\\
            q-\frac{1}{2} & \text{for}\: -q^2<p<0
        \end{cases}
    \end{equation*}
    See Figure \ref{figure1} and Figure \ref{figure2}.

    \begin{figure}[h]
        \centering
        \captionsetup{justification=centering}
        \begin{subfigure}{0.4\textwidth}
        \centering
            \includegraphics[width=\textwidth]{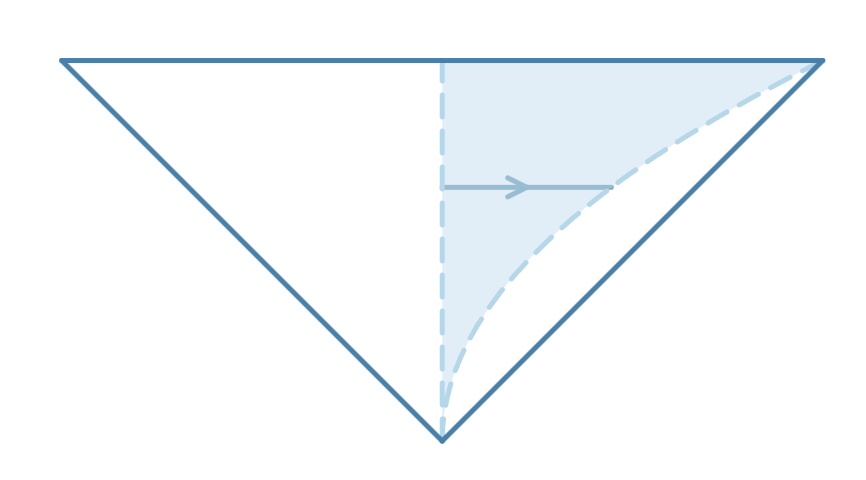}
        \end{subfigure}
        \begin{subfigure}{0.4\textwidth}
        \centering
            \includegraphics[scale=0.52]{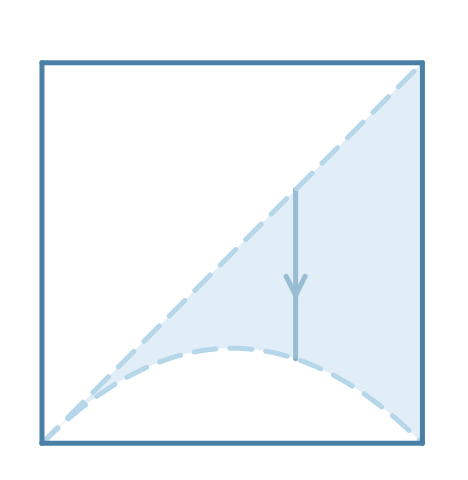}
        \end{subfigure}
        \caption{The case where $0<p<q^2$\\ Since $\xi=q-\frac{1}{2}$ and $\zeta=q-p-\frac{1}{2}$ then $-\xi^2-\frac{1}{4}<\zeta<\xi$}
        \label{figure1}
    \end{figure}

     \begin{figure}[h]
        \centering
        \captionsetup{justification=centering}
        \begin{subfigure}{0.4\textwidth}
        \centering
            \includegraphics[width=\textwidth]{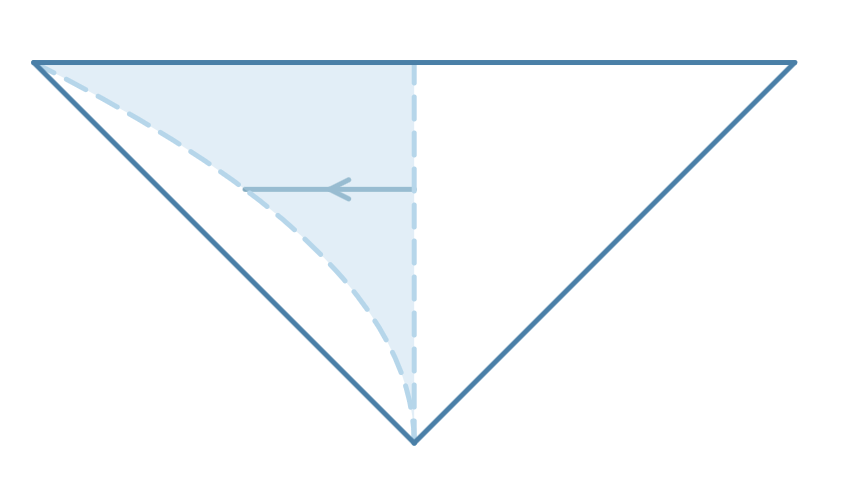}
        \end{subfigure}
        \begin{subfigure}{0.4\textwidth}
        \centering
            \includegraphics[scale=0.52]{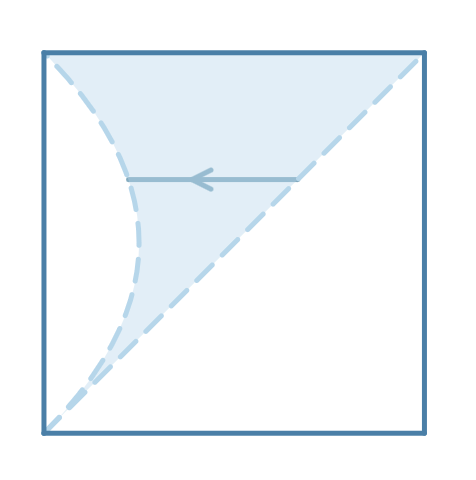}
        \end{subfigure}
        \caption{The case where $-q^2<p<0$\\ Since $\xi=p+q-\frac{1}{2}$ and $\zeta=q-\frac{1}{2}$ then $-\zeta^2-\frac{1}{4}<\xi<\zeta$}
        \label{figure2}
    \end{figure}

\end{proof}
\end{subsection}
\begin{subsection}{Case 2: $p^2\ge q^4$ }

We will use symmetric probes introduced by M. Abreu, M. Borman and D. McDuff in\cite{ABM}, generalizing the definition of probes introduced by D. McDuff in \cite{M}.

\begin{definition}
    A \textbf{probe} $P$ in a rational polytope $\Delta\subset\mathbb{R}^n$ is a directed rational line segment contained in $\Delta$ whose initial point $b_P$ lies in the interior of a facet $F_P$ of $\Delta$ and whose direction vector $v_p\in\mathbb{Z}^n$ is primitive and integrally transverse to the facet $F_P$.A probe $P$ is \textbf{symmetric} if the endpoint $e_P$ lies on the interior of a facet $F_P^\prime$ that is integrally transverse to $v_P$.
\end{definition}

In \cite{B}, J. Brendel proved that for two points $x$ and $x^\prime$ in the symmetric probe, equidistant from the boundary of the probe, the toric fibers over these two points are Hamiltonian isotopic. His result was proven for the toric manifolds. In our case we will take a probe not passing through the singularity, then the result will also hold. We give the proof for the special case used to prove our result.

\begin{proposition}
    Given the symmetric probe $\sigma=\{p=a\}$, $-1<a<1$ and $a\ne 0$, in the polytope $P_2$, let $(a, q_1), (a, q_2)\in\sigma$ be at equal distance to the boundary of the symmetric probe, then $L_2(a, q_1)$ and $L_2(a, q_2)$ are Hamiltonian isotopic.
\end{proposition}
\begin{proof}
    In \cite{OU}, J. Oakley and M. Usher gave the moment map of $\hat{F}(0)$ as $(F+G, 1-F)$ where $F(v, w)=\frac{1}{2}|v+w|$ and $G(v, w)=\frac{1}{2}(v_1+w_1)$ with the moment polytope 
    \begin{equation*}
        \{(x, y)\in\mathbb{R}^2\mid 0\le x\le 2-2y, y\ge 0\}
    \end{equation*} Recall we changed the coordinates by $p=x+y-1$ and $q=1-y$. Then the preimage $Z$ of $\sigma\cap P_2$ under the moment map is
    \begin{equation*}
        Z=\{(v, w)\in S^2\times S^2\mid v_1+w_1=2a\}
    \end{equation*}

    We will apply toric reduction to $Z$ as in \cite[Theorem 2.4]{B}. Recall the moment polytope is in $\mathfrak{t}^*\cong\mathbb{R}^2$. We denote the two generators of $\mathfrak{t}^*$ by $e_1^*$ and $e_2^*$. Let $K=\exp(e_1)\times\{1\}$. Then $K$ acts freely on $Z$. In fact, given $\theta\in K$
    \begin{equation*}
        \theta\cdot\left((v_1, v_2, v_3), (2a-v_1, w_2, w_3)\right)=\left(\begin{bmatrix}
            v_1\\
            v_2\cos(\theta)-v_3\sin(\theta)\\
            v_2\sin(\theta)+v_3\cos(\theta)
        \end{bmatrix}, \begin{bmatrix}
            2a-v_1\\
            w_2\cos(\theta)-w_3\sin(\theta)\\
            w_2\sin(\theta)+w_3\cos(\theta)
        \end{bmatrix}\right)
    \end{equation*}
    Then $\theta\cdot\left((v_1, v_2, v_3), (2a-v_1, w_2, w_3)\right)=\left((v_1, v_2, v_3), (2a-v_1, w_2, w_3)\right)$ implies that $\theta=0$ or $v_2=v_3=w_2=w_3=0$. In the later case $v_1=\pm 1$ and $2a-v_1=\pm 1$. Then $a=-1, 0, \text{or}\: 1$ contradicting with the assumption of $a$. 

    Thus $Z/K$ is a toric manifold with toric action given by $\{1\}\times\exp(e_2)$ with moment polytope $\sigma\cap P_2$. Then $Z/K$ is a sphere and the preimages of $(a, q_1)$ and $(a, q_2)$ under the moment map are two circles, denoted by $S_{q_1}$ and $S_{q_2}$ respectively. Since the two points $(a, q_1)$ and $(a, q_2)$ are at equal distance to the boundary of the symmetric probe, then $S_{q_1}$ and $S_{q_2}$ bound the disks with the same area on $Z/K$. Thus $S_{q_1}$ and $S_{q_2}$ are Hamiltonian isotopic. Then we lift the corresponding Hamiltonian to $Z$ and extend it to the whole manifold $S^2\times S^2$ by cutoff function, see \cite{AM}, \cite{B2}.
\end{proof}

Now  $L_1(p, q)$ is Hamiltonian isotopic to $L_1(p, 1-q+|p|)$. By Proposition \ref{use}, $L_1(p, 1-q+|p|)$ is Hamiltonian isotopic to 
$\begin{cases}
        T\left(\frac{1}{2}-q+p, \frac{1}{2}-q\right) & \text{for}\: q^2\le p<q\\
        T\left(\frac{1}{2}-q, \frac{1}{2}-q-p\right) & \text{for}\: -q<p\le -q^2
\end{cases}$.
As in \cite{B}, in the standard toric structure of $S^2\times S^2$, the toric fiber 
$\begin{cases}
        T\left(\frac{1}{2}-q+p, \frac{1}{2}-q\right) & \text{for}\: q^2\le p<q\\
        T\left(\frac{1}{2}-q, \frac{1}{2}-q-p\right) & \text{for}\: -q<p\le -q^2
\end{cases}$
is Hamiltonian isotopic to 
$\begin{cases}
        T\left(q-\frac{1}{2}, q-p-\frac{1}{2}\right) & \text{for}\: q^2\le p<q\\
        T\left(q+p-\frac{1}{2}, q-\frac{1}{2}\right) & \text{for}\: -q<p\le -q^2
\end{cases}$. See Figure \ref{figure3} and Figure \ref{figure4}. Thus we have proven the following proposition.

    \begin{figure}[h]
        \centering
        \captionsetup{justification=centering}
        \begin{subfigure}{0.3\textwidth}
        \centering
            \includegraphics[width=\textwidth]{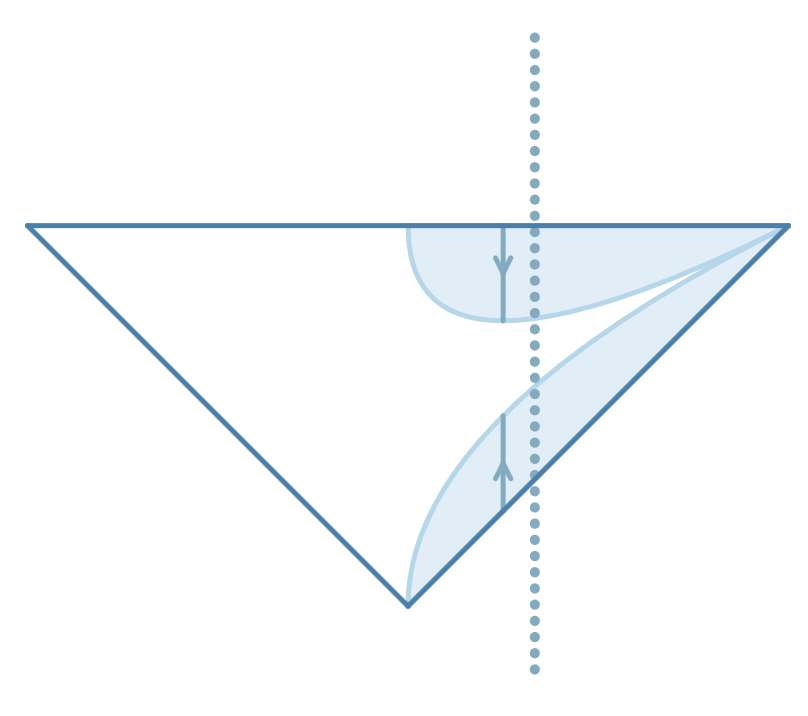}
        \end{subfigure}
        \begin{subfigure}{0.3\textwidth}
        \centering
            \includegraphics[width=\textwidth]{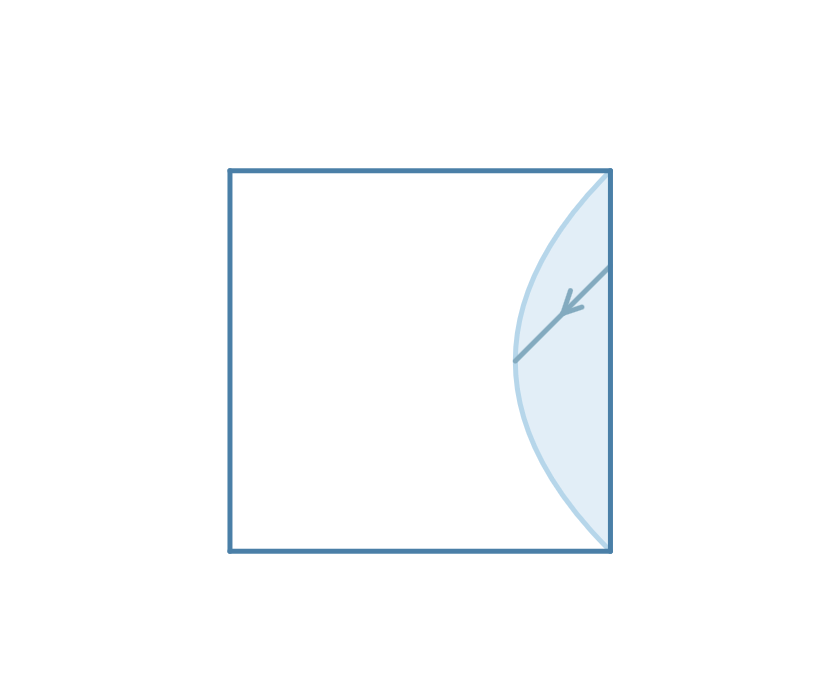}
        \end{subfigure}
        \begin{subfigure}{0.3\textwidth}
        \centering
            \includegraphics[width=\textwidth]{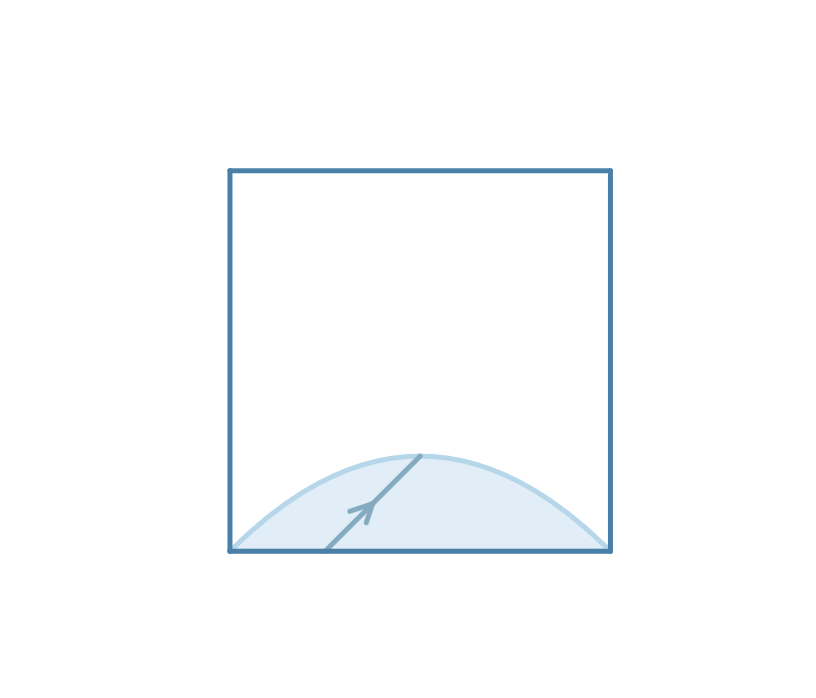}
        \end{subfigure}
        \caption{The case where $q^2\le p<q$\\ Since $\xi=q-\frac{1}{2}$ and $\zeta=q-p-\frac{1}{2}$ then $\xi-\zeta=p$}
        \label{figure3}
    \end{figure}

\begin{figure}[h]
        \centering
        \captionsetup{justification=centering}
        \begin{subfigure}{0.3\textwidth}
        \centering
            \includegraphics[width=\textwidth]{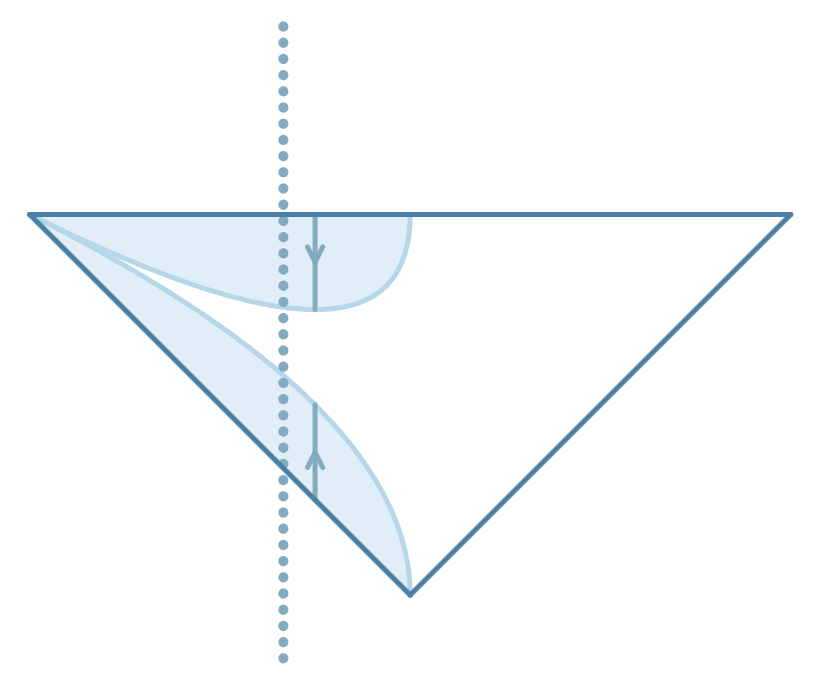}
        \end{subfigure}
        \begin{subfigure}{0.3\textwidth}
        \centering
            \includegraphics[width=\textwidth]{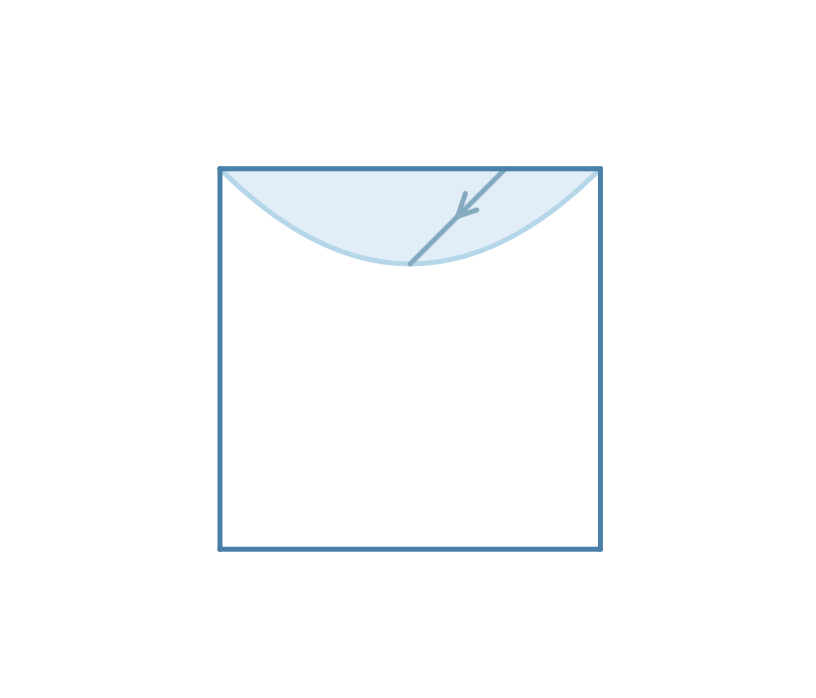}
        \end{subfigure}
        \begin{subfigure}{0.3\textwidth}
        \centering
            \includegraphics[width=\textwidth]{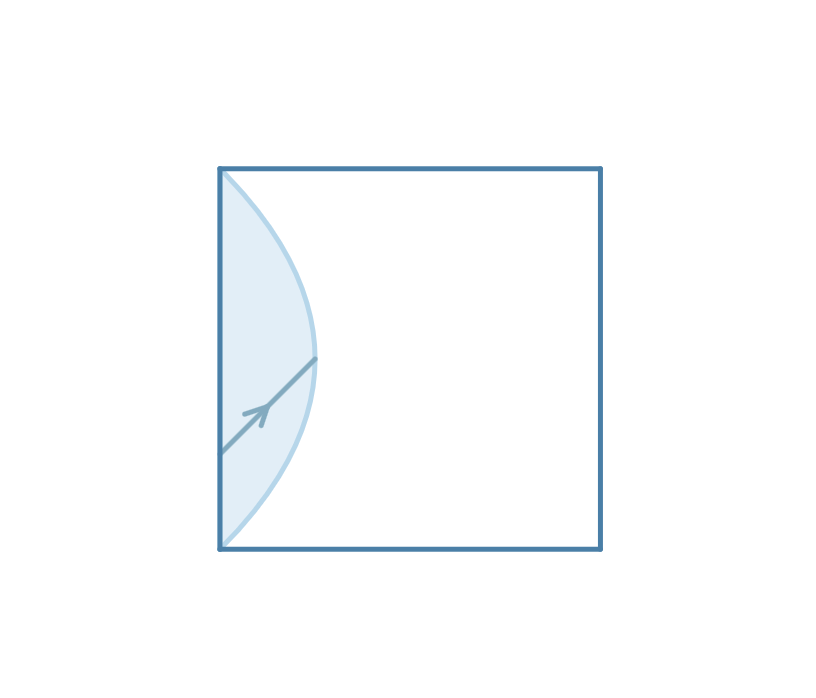}
        \end{subfigure}
        \caption{The case where $-q<p\le -q^2$\\ Since $\xi=q+p-\frac{1}{2}$ and $\zeta=q-\frac{1}{2}$ then $\xi-\zeta=p$}
        \label{figure4}
    \end{figure}

    \begin{proposition}
    The Lagrangian torus $L_1(p, q)$ is Hamiltonian isotopic to the toric fiber 
    \begin{align*}
        T(\xi, \zeta)=\begin{cases}
            T\left(q-\frac{1}{2}, q-p-\frac{1}{2}\right) & \text{for}\: q^2\le p<q\\
            T\left(q+p-\frac{1}{2}, q-\frac{1}{2}\right) & \text{for}\: -q<p\le -q^2
        \end{cases}
    \end{align*}
\end{proposition}

\end{subsection}

Finally, when we change the coordinates $(p, q)$ back to $(x, y)$, we get Theorem \ref{A}.
\end{section}

\section{Proof of Theorem \ref{B}}

In Theorem \ref{B}, $x+y=1$, which means $p=0$. In \cite{FOOO}, they have proven that $L_1(0, q)$ with $0<q\le\frac{1}{2}$ are not Hamiltonian isotopic to toric fibers for the standard toric structure. Thus we focus on $L_1(0, q)$ with $\frac{1}{2}<q<1$.  Based on Theorem \ref{A}, the moment polytope $P_1$ for the standard toric structure has been filled in by $L_1(p, q)$ with $p\ne 0$ except the diagonal. Thus if $L_1(0, q)$ is Hamiltonian isotopic to a fiber $T(\xi, \zeta)$ in the standard toric structure, it must be able to be Hamiltonian isotopic to some $T(\xi, \xi)$. Otherwise, assume $L_1(0, q)$ is Hamiltonian isotopic to some $T(\xi, \zeta)$ with $|\xi|\ne|\zeta|$. Denote the Hamiltonian isotopy by $\phi_S^t$. By Weinstein's Lagrangian neighborhood theorem, there is a symplectomorphism from a neighborhood of $T(\xi, \zeta)$ to a neighborhood of the zero section of $T^*T(\xi, \zeta)$ which takes a Lagrangian torus $C^1$-close to $T(\xi, \zeta)$ to the image of a closed $1$-form in $T^*T(\xi, \zeta)$. Denote the $1$-form corresponding to $\phi_S^1(L_1(\epsilon, q))$ by $\lambda_1$ for sufficiently small $\epsilon$ and the $1$-forms corresponding to toric fibers $T(\xi^\prime, \zeta^\prime)$ by $\lambda_{\xi^\prime-\xi, \zeta^\prime-\zeta}$ for $(\xi^\prime,\zeta^\prime)$ close enough to $(\xi, \zeta)$. Then there is toric fiber $T(\xi^\prime, \zeta^\prime)$ such that $[\lambda_{\xi^\prime-\xi, \zeta^\prime-\zeta}]=[\lambda_1]\in H^1(T(\xi, \zeta), \mathbb{R})$.
Thus there is a smooth function $h: T(\xi, \zeta)\to\mathbb{R}$ such that $\lambda_{\xi^\prime-\xi, \zeta^\prime-\zeta}-\lambda_1=dh$. Then $\Pi\circ h$ generates a Hamiltonian isotopy between $\phi_S^1(L_1(\epsilon, q))$ and $T(\xi^\prime, \zeta^\prime)$ where $\Pi: T^*T(\xi, \zeta)\to T(\xi, \zeta)$ is the projection.
On the other hand, $L_1(\epsilon, q)$ is Hamiltonian isotopic to $T\left(q-\frac{1}{2}, q-\epsilon-\frac{1}{2}\right)$ for positive $\epsilon$ by Theorem \ref{A}. If $\epsilon$ is small enough, $T(\xi^\prime, \zeta^\prime)$ is not Hamiltonian isotopic to $T\left(q-\frac{1}{2}, q-\epsilon-\frac{1}{2}\right)$ by \cite{B}. We get a contradiction.

Then we will compute the displacement energy germ introduced in \cite{CS} to show that it is not possible that $L_1(0, q)$ is Hamiltonian isotopic to some $T(\xi, \xi)$.

\begin{definition}[\cite{CS}]
    Let $(M,\omega)$ be a symplectic manifold and $L$ be a closed embedded Lagrangian submanifold. The \textbf{displacement energy germ} is a function germ $S_L^e: H^{1}(L, \mathbb{R})\to [0, \infty]$ at the point $0\in H^{1}(L, \mathbb{R})$ defined as $S_L^e(\delta)=e(L_\delta)$ where $L_{\delta}$ is the image of a closed $1$-form on $L$ representing a sufficently small class $\delta\in H^1(L, \mathbb{R})$ and $e(L_\delta)$ is the displacement energy of $L_\delta$.
\end{definition}

\begin{proposition}[\cite{CS}]\label{germ}
    For each symplectomorphism $\psi$ we have $S_{\psi(L)}^e=S_L^e\circ(\psi|_L)^*$.
\end{proposition}

Note that the displacement energy of $T(\xi, \zeta)$ with $(\xi, \zeta)\ne (0, 0)$ is $\min\left\{\frac{1}{2}-|\xi|, \frac{1}{2}-|\zeta|\right\}$ since the displacement energy of the circle $v_1=a(\ne 0)$ in $S^2$ is $\min\left\{\frac{1}{2}-\frac{|a|}{2}\right\}$. Also see \cite[Example 4.1]{B2}.

Now we compute the displacement energy germ of $L_1(0, q)$. Since $H^{1}(L_1(0, q), \mathbb{R})\cong H_1(L_1(0, q), \mathbb{R})$, then a neighborhood of $0\in H^{1}(L_1(0, q), \mathbb{R})$ can be identified with a neighborhood of the point $(0, q)$ in the moment polytope $P_2$. Let $(\delta_1, \delta_2)\in\mathbb{R}^2$ be close enough to $(0, 0)$ and $\delta_1\ne 0$. Then $S_{L_1(0, q)}^e(\delta_1, \delta_2)=e(L_1(\delta_1, q+\delta_2))$. 

First we consider the case where $\delta_1>0$. By Theorem \ref{A}, $L_1(\delta_1, q+\delta_2)$ is Hamiltonian isotopic to the toric fiber $
    T\left(q+\delta_2-\frac{1}{2}, q+\delta_2-\delta_1-\frac{1}{2}\right) $. Thus the displacement energy is 
    \begin{equation*}
        e(L_1(\delta_1, q+\delta_2))=\min\left\{\frac{1}{2}-\left|q+\delta_2-\frac{1}{2}\right|, \frac{1}{2}-\left|q+\delta_2-\delta_1-\frac{1}{2}\right|\right\}
    \end{equation*}
Since $q>\frac{1}{2}$, we can choose $\delta_1$ and $\delta_2$ small enough such that $q+\delta_2-\frac{1}{2}>0$ and $q+\delta_2-\delta_1-\frac{1}{2}>0$. Thus the displacement energy is
 \begin{align*}
        e(L_1(\delta_1, q+\delta_2)) &=\min\left\{\frac{1}{2}-\left(q+\delta_2-\frac{1}{2}\right), \frac{1}{2}-\left(q+\delta_2-\delta_1-\frac{1}{2}\right)\right\}\\
        &=\min\left\{1-q-\delta_2, 1-q-\delta_2+\delta_1\right\}\\
        &=1-q-\delta_2
    \end{align*}
    The last equality is from $\delta_1>0$.
    
Then we consider the case where $\delta_1<0$. By Theorem \ref{A}, $L_1(\delta_1, q+\delta_2)$ is Hamiltonian isotopic to the toric fiber $T\left(q+\delta_1+\delta_2-\frac{1}{2}, q+\delta_2-\frac{1}{2}\right)$. The displacement energy is
\begin{align*}
    e(L_1(\delta_1, q+\delta_2)) &=\min\left\{\frac{1}{2}-\left(q+\delta_1+\delta_2-\frac{1}{2}\right), \frac{1}{2}-\left(q+\delta_2-\frac{1}{2}\right)\right\}\\
    &=\min\left\{1-q-\delta_1-\delta_2, 1-q-\delta_2\right\}\\
    &=1-q-\delta_2
\end{align*}
The last equality is from $\delta_1<0$

Thus $S_{L_1(0, q)}^e(\delta_1, \delta_2)=1-q-\delta_2$ when $\delta_1\ne 0$

Next we compute the displacement energy germ of $T(\xi, \xi)$. Let $(\delta_1^\prime, \delta_2^\prime)\in\mathbb{R}^2$ be close enough to $(0, 0)$. If $\xi>0$, then $\xi+\delta_1^\prime>0$ and $\xi+\delta_2^\prime>0$ for small enough $\delta_1^\prime$ and $\delta_2^\prime$. Thus $S_{T(\xi, \xi)}^e(\delta_1^\prime, \delta_2^\prime)=e(T(\xi+\delta_1^\prime, \xi+\delta_2^\prime))=\min\left\{\frac{1}{2}-\xi-\delta_1^\prime, \frac{1}{2}-\xi-\delta_2^\prime\right\}$. If $\xi<0$, then $\xi+\delta_1^\prime<0$ and $\xi+\delta_2^\prime<0$ for small enough $\delta_1^\prime$ and $\delta_2^\prime$. Thus $S_{T(\xi, \xi)}^e(\delta_1^\prime, \delta_2^\prime)=e(T(\xi+\delta_1^\prime, \xi+\delta_2^\prime))=\min\left\{\frac{1}{2}+\xi+\delta_1^\prime, \frac{1}{2}+\xi+\delta_2^\prime\right\}$. 

The displacement energy germ of $T(\xi, \xi)$ is determined by two linearly independent functions but the displacement energy germ of $L_1(0, q)$ is determine by a single function when $\delta_1\ne 0$. Thus there is not a linear isomorphism on $\mathbb{R}^2$ taking $S_{L_1(0, q)}^e(\delta_1, \delta_2)$ to $S_{T(\xi, \xi)}^e(\delta_1^\prime, \delta_2^\prime)$. Thus $L(0, q)$ and $T(\xi, \xi)$ are not symplectomorphic, in particular, not Hamiltonian isotopic.

\section{Appendix}

\begin{lemma}\label{norm}
    Let $z$ be a point in the curve $\Gamma$. Then
    \begin{equation*}
        ||z||=\sqrt{\frac{1}{4}\left(1-\frac{\sqrt{q^2-p^2}\sqrt{1-q^2}}{q}\sin(\theta)\right)^2-\frac{1}{4}p^2}
    \end{equation*}
\end{lemma}
\begin{proof}
   First we write down the explicit expression of the curve $\Gamma$
\begin{align*}
    \Gamma =\left\{\frac{1-2q^2+p^2-\frac{(q^2-p^2)(1-q^2)}{q^2}\sin^2(\theta)+2i\sqrt{q^2-p^2}\sqrt{1-q^2}\cos(\theta)}{2\sqrt{\left(1+\frac{\sqrt{q^2-p^2}\sqrt{1-q^2}}{q}\sin(\theta)\right)^2-p^2}}\:\middle|\: 0\le\theta\le 2\pi\right\}
\end{align*}
    For a point $z=(x, y)$ in $\Gamma$,
    \begin{align*}
        x &=\frac{1-2q^2+p^2-\frac{(q^2-p^2)(1-q^2)}{q^2}\sin^2(\theta)}{2\sqrt{\left(1+\frac{\sqrt{q^2-p^2}\sqrt{1-q^2}}{q}\sin(\theta)\right)^2-p^2}}\\
        &=-\frac{1}{2}\sqrt{\left(1+\frac{\sqrt{q^2-p^2}\sqrt{1-q^2}}{q}\sin(\theta)\right)^2-p^2}+\frac{1-q^2+\frac{\sqrt{q^2-p^2}\sqrt{1-q^2}}{q}\sin(\theta)}{\sqrt{\left(1+\frac{\sqrt{q^2-p^2}\sqrt{1-q^2}}{q}\sin(\theta)\right)^2-p^2}}
    \end{align*}
    \begin{align*}
        x^2 &=\frac{1}{4}\left(\left(1+\frac{\sqrt{q^2-p^2}\sqrt{1-q^2}}{q}\sin(\theta)\right)^2-p^2\right)
        +\frac{\left(1-q^2+\frac{\sqrt{q^2-p^2}\sqrt{1-q^2}}{q}\sin(\theta)\right)^2}{\left(1+\frac{\sqrt{q^2-p^2}\sqrt{1-q^2}}{q}\sin(\theta)\right)^2-p^2}\\ &-1+q^2-\frac{\sqrt{q^2-p^2}\sqrt{1-q^2}}{q}\sin(\theta)
    \end{align*}
    and
    \begin{equation*}
        y=\frac{2\sqrt{q^2-p^2}\sqrt{1-q^2}\cos(\theta)}{2\sqrt{\left(1+\frac{\sqrt{q^2-p^2}\sqrt{1-q^2}}{q}\sin(\theta)\right)^2-p^2}}
    \end{equation*}
    \begin{equation*}
        y^2=\frac{(q^2-p^2)(1-q^2)\cos^2(\theta)}{\left(1+\frac{\sqrt{q^2-p^2}\sqrt{1-q^2}}{q}\sin(\theta)\right)^2-p^2}
    \end{equation*}
Note that 
\begin{equation*}
    \frac{\left(1-q^2+\frac{\sqrt{q^2-p^2}\sqrt{1-q^2}}{q}\sin(\theta)\right)^2}{\left(1+\frac{\sqrt{q^2-p^2}\sqrt{1-q^2}}{q}\sin(\theta)\right)^2-p^2}+\frac{(q^2-p^2)(1-q^2)\cos^2(\theta)}{\left(1+\frac{\sqrt{q^2-p^2}\sqrt{1-q^2}}{q}\sin(\theta)\right)^2-p^2}=1-q^2
\end{equation*}
Then 
\begin{equation*}
    x^2+y^2=\frac{1}{4}\left(1-\frac{\sqrt{q^2-p^2}\sqrt{1-q^2}}{q}\sin(\theta)\right)^2-\frac{1}{4}p^2
\end{equation*}
\end{proof}

\begin{proposition}
    The $\omega^p$-area enclosed by $\Gamma$ is $2\pi-2\pi q$ for $0<p^2<q^4$ .
\end{proposition}
\begin{proof}
    First we determine a $1$-form $\sigma$ such that $d\sigma=\omega^p$. Note that $\sigma$ can be arranged to have the form\\ $\left(\frac{\sqrt{p^2+4r^2}}{2}+C\right)d\phi$. On one hand we can compute 
    \begin{equation*}
    \int_{B^2(1)}\frac{2r}{\sqrt{p^2+4r^2}}dr\wedge d\phi=\pi\sqrt{p^2+4}-\pi|p|
    \end{equation*}
    On the other hand,
    \begin{equation*}
        \int_{B^2(1)}\frac{2r}{\sqrt{p^2+4r^2}}dr\wedge d\phi=\int_{S^1}\left(\frac{\sqrt{p^2+4\cdot 1^2}}{2}+C\right)d\phi=\pi\sqrt{p^2+4}+2\pi C
    \end{equation*}
    Thus $C=-\frac{|p|}{2}$ and $\sigma=\left(\frac{\sqrt{p^2+4r^2}}{2}-\frac{|p|}{2}\right)d\phi$.

    Now we compute $\int_\Gamma\sqrt{p^2+4r^2}d\phi=\int_\Gamma\sqrt{p^2+4r^2}\frac{d\phi}{d\theta}d\theta$. Note that $1-\frac{\sqrt{q^2-p^2}\sqrt{1-q^2}}{q}\sin(\theta)>0$. By Lemma \ref{norm}, we have 
    \begin{equation*}
        \sqrt{p^2-4r^2}=1-\frac{\sqrt{q^2-p^2}\sqrt{1-q^2}}{q}\sin(\theta)
    \end{equation*}
    By taking the derivative with respect to $\theta$ on both sides of 
    \begin{equation*}
        \tan(\phi)=\frac{y}{x}=\frac{2\sqrt{q^2-p^2}\sqrt{1-q^2}\cos(\theta)}{1-2q^2+p^2-\frac{(q^2-p^2)(1-q^2)}{q^2}\sin^2(\theta)}
    \end{equation*}
    we can get 
    \begin{align*}
        \frac{d\phi}{d\theta} &=\frac{\frac{-2\sqrt{q^2-p^2}\sqrt{1-q^2}\sin(\theta)}{q}}{\left[\left(1+\frac{\sqrt{q^2-p^2}\sqrt{1-q^2}}{q}\sin(\theta)\right)^2-p^2\right]\cdot\left[\left(1-\frac{\sqrt{q^2-p^2}\sqrt{1-q^2}}{q}\sin(\theta)\right)^2-p^2\right]}\\    &\cdot\left(\frac{p^2-q^2+p^2(1-q^2)}{q}+\frac{(q^2-p^2)(1-q^2)}{q}\sin^2(\theta)\right)\\
        &=-\frac{1}{2}\left(\frac{1}{\left(1-\frac{\sqrt{q^2-p^2}\sqrt{1-q^2}}{q}\sin(\theta)\right)^2-p^2}-\frac{1}{\left(1+\frac{\sqrt{q^2-p^2}\sqrt{1-q^2}}{q}\sin(\theta)\right)^2-p^2}\right)\\
        &\cdot\left(\frac{p^2-q^2+p^2(1-q^2)}{q}+\frac{(q^2-p^2)(1-q^2)}{q}\sin^2(\theta)\right)\\
    \end{align*}
    Now 
    \begin{align*}
        &\int_\Gamma\sqrt{p^2+4r^2}d\phi\\
        &=\int_0^{2\pi}-\frac{1}{2}\frac{1-\frac{\sqrt{q^2-p^2}\sqrt{1-q^2}}{q}\sin(\theta)}{\left(1-\frac{\sqrt{q^2-p^2}\sqrt{1-q^2}}{q}\sin(\theta)\right)^2-p^2}\cdot\left(\frac{p^2-q^2+p^2(1-q^2)}{q}+\frac{(q^2-p^2)(1-q^2)}{q}\sin^2(\theta)\right)d\theta\\
        &-\int_0^{2\pi}-\frac{1}{2}\frac{1-\frac{\sqrt{q^2-p^2}\sqrt{1-q^2}}{q}\sin(\theta)}{\left(1+\frac{\sqrt{q^2-p^2}\sqrt{1-q^2}}{q}\sin(\theta)\right)^2-p^2}\cdot\left(\frac{p^2-q^2+p^2(1-q^2)}{q}+\frac{(q^2-p^2)(1-q^2)}{q}\sin^2(\theta)\right)d\theta
    \end{align*}
    In the first integral we replace $\theta$ with $2\pi-\theta$. Then 
    \begin{align*}
        &\int_\Gamma\sqrt{p^2+4r^2}d\phi\\
        &=-\frac{1}{2}\int_0^{2\pi}\left(\frac{p^2-q^2+p^2(1-q^2)}{q}+\frac{(q^2-p^2)(1-q^2)}{q}\sin^2(\theta)\right)\frac{2\frac{\sqrt{q^2-p^2}\sqrt{1-q^2}}{q}\sin(\theta)}{\left(1+\frac{\sqrt{q^2-p^2}\sqrt{1-q^2}}{q}\sin(\theta)\right)^2-p^2}d\theta
    \end{align*}
   Note that 
   \begin{align*}
       &\frac{p^2-q^2+p^2(1-q^2)}{q}+\frac{(q^2-p^2)(1-q^2)}{q}\sin^2(\theta)\\
       &=q\left[\left(\frac{\sqrt{q^2-p^2}\sqrt{1-q^2}}{q}\sin(\theta)+1\right)^2-p^2-2\frac{q^2-p^2}{q^2}-2\frac{\sqrt{q^2-p^2}\sqrt{1-q^2}}{q}\sin(\theta)\right]
   \end{align*}
   Then 
   \begin{align*}
       &\int_\Gamma\sqrt{p^2+4r^2}d\phi\\
       &=-\frac{q}{2}\int_0^{2\pi}2\frac{\sqrt{q^2-p^2}\sqrt{1-q^2}}{q}\sin(\theta)d\theta\\
       &+q\int_0^{2\pi}\left(\frac{q^2-p^2}{q^2}+\frac{\sqrt{q^2-p^2}\sqrt{1-q^2}}{q}\sin(\theta)\right)\frac{2\frac{\sqrt{q^2-p^2}\sqrt{1-q^2}}{q}\sin(\theta)}{\left(1+\frac{\sqrt{q^2-p^2}\sqrt{1-q^2}}{q}\sin(\theta)\right)^2-p^2}d\theta\\
       &=q\int_0^{2\pi}\left(\frac{q^2-p^2}{q^2}+\frac{\sqrt{q^2-p^2}\sqrt{1-q^2}}{q}\sin(\theta)\right)\frac{2\frac{\sqrt{q^2-p^2}\sqrt{1-q^2}}{q}\sin(\theta)}{\left(1+\frac{\sqrt{q^2-p^2}\sqrt{1-q^2}}{q}\sin(\theta)\right)^2-p^2}d\theta
   \end{align*}
   since $\int_0^{2\pi}2\frac{\sqrt{q^2-p^2}\sqrt{1-q^2}}{q}\sin(\theta)d\theta=0$.
   Note that 
   \begin{align*}
       &\left(\frac{q^2-p^2}{q^2}+\frac{\sqrt{q^2-p^2}\sqrt{1-q^2}}{q}\sin(\theta)\right)\frac{\sqrt{q^2-p^2}\sqrt{1-q^2}}{q}\sin(\theta)\\
       &=\left(\frac{\sqrt{q^2-p^2}\sqrt{1-q^2}}{q}\sin(\theta)+1\right)^2-p^2-\frac{p^2+q^2}{q^2}\frac{\sqrt{q^2-p^2}\sqrt{1-q^2}}{q}\sin(\theta)-1+p^2
   \end{align*}
   We have 
   \begin{align*}
       \int_\Gamma\sqrt{p^2+4r^2}d\phi
       &=2q\int_0^{2\pi}1d\theta+2q\int_0^{2\pi}\frac{-\frac{p^2+q^2}{q^2}\frac{\sqrt{q^2-p^2}\sqrt{1-q^2}}{q}\sin(\theta)-1+p^2}{\left(1+\frac{\sqrt{q^2-p^2}\sqrt{1-q^2}}{q}\sin(\theta)\right)^2-p^2}d\theta\\
       &=4\pi q-2q\int_0^{2\pi}\frac{\frac{p^2+q^2}{q^2}\frac{\sqrt{q^2-p^2}\sqrt{1-q^2}}{q}\sin(\theta)+1-p^2}{\left(1+\frac{\sqrt{q^2-p^2}\sqrt{1-q^2}}{q}\sin(\theta)\right)^2-p^2}d\theta\\
   \end{align*}
   Now we focus on the integral
   \begin{align*}
       &\int_0^{2\pi}\frac{\frac{p^2+q^2}{q^2}\frac{\sqrt{q^2-p^2}\sqrt{1-q^2}}{q}\sin(\theta)+1-p^2}{\left(1+\frac{\sqrt{q^2-p^2}\sqrt{1-q^2}}{q}\sin(\theta)\right)^2-p^2}d\theta\\
       &=\frac{(1+p)(q^2+p)}{2q^2}\int_0^{2\pi}\frac{1}{1+p+\frac{\sqrt{q^2-p^2}\sqrt{1-q^2}}{q}\sin(\theta)}d\theta+\frac{(1-p)(q^2-p)}{2q^2}\int_0^{2\pi}\frac{1}{1-p+\frac{\sqrt{q^2-p^2}\sqrt{1-q^2}}{q}\sin(\theta)}d\theta
   \end{align*}
One can easily compute
\begin{equation*}
\int_0^{2\pi}\frac{1}{1+p+\frac{\sqrt{q^2-p^2}\sqrt{1-q^2}}{q}\sin(\theta)}d\theta= \frac{2\pi q}{q^2+p}
\end{equation*}
and 
\begin{equation*}
\int_0^{2\pi}\frac{1}{1-p+\frac{\sqrt{q^2-p^2}\sqrt{1-q^2}}{q}\sin(\theta)}d\theta= \frac{2\pi q}{q^2-p}
\end{equation*}
Thus 
\begin{equation*}
    \int_0^{2\pi}\frac{\frac{p^2+q^2}{q^2}\frac{\sqrt{q^2-p^2}\sqrt{1-q^2}}{q}\sin(\theta)+1-p^2}{\left(1+\frac{\sqrt{q^2-p^2}\sqrt{1-q^2}}{q}\sin(\theta)\right)^2-p^2}d\theta= \frac{2\pi}{q}
\end{equation*}
Then 
\begin{equation*}
    \int_\Gamma\sqrt{p^2+4r^2}d\phi=4\pi q-4\pi
\end{equation*}

Next one can follow the same process to show that 
\begin{equation*}
    \int_\Gamma\frac{|p|}{2}d\phi=0
\end{equation*}

Thus if the curve $\Gamma$ is reparametrized by replacing $\theta$ with $2\pi-\theta$, then the $\omega^p$-area enclosed by $\Gamma$ is $2\pi-2\pi q$.
\end{proof}

\begin{remark}
    If we write the $1$-form $\frac{\sqrt{p^2+4r^2}}{2}d\phi$ as $\frac{-y\sqrt{p^2+4x^2+4y^2}}{2(x^2+y^2)}dx+\frac{x\sqrt{p^2+4x^2+4y^2}}{2(x^2+y^2)}dy$, then it is easy to see that $\frac{\sqrt{p^2+4r^2}}{2}d\phi$ is not defined at $(0,0)$. On the other hand $\left(\frac{\sqrt{p^2+4r^2}}{2}-\frac{|p|}{2}\right)d\phi=\frac{-2y}{\sqrt{p^2+4x^2+4y^2}+|p|}dx+\frac{2x}{\sqrt{p^2+4x^2+4y^2}+|p|}dy$ is defined at $(0, 0)$.
\end{remark}

\bibliographystyle{amsalpha}
\bibliography{ref}

\end{document}